\documentclass[12pt]{amsart}

\usepackage{amsmath, amsthm, amssymb}
\usepackage{url} 
\usepackage{graphicx}
\usepackage{xcolor}
\usepackage{moreverb}

\usepackage{fancyvrb}

\def\r{\mathbb R}

\def\s{\mathbb S}
\def\h{\mathbb H}

\usepackage{listings}
 \lstnewenvironment{code}[1][]
  {\lstset{#1}}
  {}

\newtheorem{theorem}{Theorem}[section]

\theoremstyle{definition}
\newtheorem{definition}[theorem]{Definition}
\newtheorem{example}[theorem]{Example}
\newtheorem{remark}[theorem]{Remark}

\begin{document}

\title[Anti-torqued slant helices and Torqued Curves]{Anti-torqued slant helices and Torqued Curves in Riemannian manifolds}

\author{Muhittin Evren Aydin$^1$}
\address{$^1$Department of Mathematics, Faculty of Science, Firat University, Elazig,  23200 Turkey
\newline
ORCID: 0000-0001-9337-8165}
\email{meaydin@firat.edu.tr}
\author{ Adela Mihai$^2$}
 \address{$^2$Technical University of Civil Engineering Bucharest,
Department of Mathematics and Computer Science, 020396, Bucharest, Romania
and Transilvania University of Bra\c{s}ov, Interdisciplinary Doctoral
School, 500036, Bra\c{s}ov, Romania
\newline
ORCID: 0000-0003-2033-8394}

 \email{adela.mihai@utcb.ro, adela.mihai@unitbv.ro}
\author{ Cihan Özgür$^3$}
 \address{$^3$Department of Mathematics, Izmir Democracy University, 35140, Karabaglar, Izmir, Turkey
 \newline
ORCID: 0000-0002-4579-7151}
 \email{cihan.ozgur@idu.edu.tr}

\keywords{Anti-torqued vector field, torqued vector field, slant helix, general helix, Riemannian manifold}
\subjclass{53B20; 53A35}
\begin{abstract}
In this paper, we introduce the notion of an anti-torqued slant helix in a Riemannian manifold, defined as a curve whose principal vector field makes a constant angle with an anti-torqued vector field globally defined on the ambient manifold. We characterize and classify such curves through systems of differential equations involving their invariants. Several illustrative examples are also provided. Finally, we study torqued curves, defined as curves for which the inner product function of the principal vector field and a torqued vector field along the curve is a given constant.
\end{abstract}
\maketitle

\section{Introduction} \label{intro}

Let $M$ be a submanifold of a Riemannian manifold $\widetilde{M}$. Denote by $T_pM$ and $N_pM$ the tangent and the normal spaces of $M$ in $\widetilde{M}$ at some point $p\in M$, respectively. The {\it first normal space} of $M$ at $p\in M$ is a subspace of $N_pM$ defined by
$$
\mbox{Im }h_p =\text{Span}\{h(X,Y): X,Y \in T_pM \},
$$
where $h$ is the second fundamental form of $M$ in $\widetilde{M}$.

Let $V$ be a nowhere-zero vector field on  $\widetilde{M}$. Chen \cite{crs0,crs1,crs3} defined $M$ as a {\it rectifying submanifold} in $\widetilde{M}$ with respect to $V$ if the normal component $V^{\perp} \neq 0$ on $M$ and
\begin{equation} \label{rs}
\langle V(p),\mbox{Im } h_p\rangle=0, \quad \forall p\in M,
\end{equation}
where $\langle , \rangle$ is the metric tensor on $\widetilde{M}$. Chen \cite{crs0, crs1}, in collaboration with Oh \cite{crs2}, completely classified these submanifolds in both the Euclidean and pseudo-Euclidean cases. Moreover, these submanifolds were studied in the multiplicative Euclidean space \cite{ahy}.

Let $\widetilde{M}$ denote one of the complete simply-connected models of Riemannian manifolds with constant sectional curvature and, let $V$ be a concircular vector field on $\widetilde{M}$. Recently, Lucas and Ortega-Yag\"ues \cite{loy4,loy5} defined $M$ as a {\it concircular manifold} in $\widetilde{M}$ with axis $V$ if the function $\theta:=\langle V, N\rangle$ remains constant on $M$, where $N(p) \in  \mbox{Im }h_p$, $p\in M$, is any unitary vector field.

With this definition, we have
$$
\langle V,h(X,Y)\rangle=|h(X,Y)|\theta ,\quad X,Y \in\mathfrak{X}(M).
$$
Since $h$ is a $(2,1)$-tensor field, if $\theta \neq 0$, then $M$ must be totally geodesic in $\widetilde{M}$ whenever both the dimension and the codimension of $M$ are greater than $1$. Otherwise, i.e. $\theta =0$, this notion coincides with the definition of a rectifying submanifold. However, if $M$ is a curve or a hypersurface in $\widetilde{M}$, then $\theta$ can be an arbitrary constant, as the first normal space of $M$ is of one-dimensional. In other words, concircular manifolds can only be either concircular curves or concircular hypersurfaces or rectifying submanifolds.

The authors in \cite{loy4,loy5} classified concircular curves and hypersurfaces in $\widetilde{M}$. Concircular curves are also known as {\it concircular helices}. When the axis is a torqued vector field or a concircular vector field, rectifying submanifolds were characterized in \cite{crs3}. Another classification of rectifying submanifolds in Riemannian manifolds with an anti-torqued axis was done by the authors of the present paper \cite{amc1}.

Consider a regular curve $\gamma:I\subset \r\to \r^m$ given by $\gamma=\gamma(s)$, where $s$ is the arc-length parameter. Recall that a concircular vector field $V$ on $\r^m$ is either parallel or is given by (see \cite{cbook})
$$
x \mapsto V(x)=\rho x+\vec{v}, \quad \rho\in \r, \quad \vec{v}\in \r^m.
$$
If $\gamma$ is a concircular helix, then it follows that
$$
\langle \rho \gamma(s)+\vec{v} ,N^\gamma(s) \rangle = \theta, \quad \forall s \in I,
$$
where $N^\gamma(s)$ is the unit principal normal vector of $\gamma$, and $\theta \in \r$. A concircular helix is said to be {\it proper} if both $\rho$ and $\theta$ are nonzero.

Depending on the choice of the constants $\rho$ and $\theta$, concircular helices reduce to well-known classes of curves. More precisely, if the concircular factor $\rho$ is zero, corresponding to the case that $V$ is a parallel vector field, then $\gamma$ belongs to the family of helices \cite{bar,it,loy2,loy3}. On the other hand, if the constant $\theta$ is zero for every $s\in I$, then $\gamma$ is congruent to a rectifying curve (see \cite{cgb}). Rectifying curves, introduced by Chen in 2003 (see \cite{crc0}), have been a topic of great interest. For example, see \cite{cgb,crc1, crc2, dca, inpt, in, jamr, loy0, loy1}.

Similar to the concept of concircular helices, Ateş et al. \cite{agy} introduced $N$-slant curves in Sasakian 3-manifolds, defined as curves whose principal normal vector field makes a constant angle with the Reeb vector field. The geometric properties of these curves were studied in relation to the mean curvature vector field.

We also wish to highlight some recent studies on curves and surfaces in arbitrary Riemannian manifolds. In \cite{dfds}, da Silva et al. characterized and classified curves and surfaces in arbitrary 3-dimensional Riemannian manifolds that make a constant angle with a parallel transported vector field, that is, a vector field $V$ satisfying $\widetilde{\nabla}_XV=0$, for every tangent $X$ to the submanifolds. These curves may have interdisciplinary applications, for example in roads design, a study already started in \cite{bm}. In another work, Çalışkan and Şahin \cite{casn} characterized slant helices in $m$-dimensional Riemannian manifolds by means of differential equations. Moreover, the authors in \cite{casn} provided criteria under which circles along immersed submanifolds can be characterized as slant helices, and conversely.

Motivated by these studies, we introduce the following:

\begin{definition} \label{def}
Let $\gamma=\gamma(s)$, $s\in I$, be a Frenet curve of order $r$, $1 \leq r \leq m$, in a Riemannian manifold $\widetilde{M}$ ($\mbox{dim}(\widetilde{M})=m$) endowed with an anti-torqued vector field $V$. The curve $\gamma$ is called an {\it anti-torqued slant helix} if
$$
\langle V|_{\gamma(s)} ,N^\gamma (s) \rangle =\cos \theta, \quad \forall s \in I, \quad \theta \in [0,2\pi],
$$
where $V$ is referred to be the {\it axis} of $\gamma$, and $N^\gamma (s)$ is the unit principal vector field of $\gamma$.
\end{definition}

Moreover, we study curves in $\widetilde{M}$ for which the function $\langle V|_{\gamma(s)} ,N(s) \rangle = \theta \in \r$ is a given constant, where $V$ is a torqued vector field globally defined on $\widetilde{M}$. Following \cite{loy4, loy5}, we call these curves {\it torqued curves}.

Special vector fields play an important role in the study of geometry and topology of the ambient spaces and are of great interest (see \cite{bo,bo1,cs,c00,cv,dmtv,dta,id,ya0,ymy}). Moreover they also have some applications in physics \cite{c0,dr,mam0,mam1,mam3}.

This paper is organized as follows: In Section \ref{pre}, the fundamental invariants of curves on Riemannian manifolds are given, and in Section \ref{tors}, torse-forming vector fields are revisited with examples. In Section \ref{anti-helices}, anti-torqued slant helices in Riemannian manifolds are characterized and classified (Theorems \ref{anti-slant}, \ref{anti-slant-3d}, and \ref{anti-slant2}). Non-trivial examples are presented (Examples \ref{exm-zero-tors} and \ref{exm-nozero-tors}). Finally we characterize torqued curves in Section \ref{torq-curv} (Theorems \ref{tor-slant} and \ref{conc}).

\section{Preliminaries} \label{pre}

Let $\gamma=\gamma(t)$, $t\in I$ be a regular curve in an $m$-dimensional Riemannian manifold $\widetilde{M}$, and let $T^\gamma$ be its tangent vector field. The curve $\gamma$ is said to be in {\it general position up to the order} $1\leq r \leq m$ (see \cite{lmm}) if the set
$$
\{T^\gamma,\widetilde{\nabla}_{T^\gamma}T^\gamma,...,\widetilde{\nabla}_{T^\gamma}^{r-1}T^\gamma\}
$$
is linearly independent for every $s\in I$, where $\widetilde{\nabla}$ is a linear connection. In addition, when $\gamma$ is in general position up to the order $1\leq r \leq m$ and $\widetilde{\nabla}$ is the Levi-Civita connection, then it is called a {\it Frenet curve of order} $r$.

Assume that $\gamma$ is a Frenet curve of order $m$ and that $\widetilde{M}$ is an oriented connected $m$-dimensional Riemannian manifold. Then there exists a unique set $\{N_1^\gamma,...,N_m^\gamma\}$ of smooth vector fields along $\gamma$, and smooth functions $\kappa_j:I\to \r$, with $\kappa_j >0$ and $0\leq j \leq m-1$, such that
\begin{enumerate}
\item $(N_1^\gamma,...,N_m^\gamma)$ is a positively oriented orthonormal frame along $\gamma$,

\item $(N_1^\gamma,...,N_m^\gamma)$ and $
(T^\gamma,\widetilde{\nabla}_{T^\gamma}T^\gamma,...,\widetilde{\nabla}_{T^\gamma}^{r-1}T^\gamma)
$ span the same vector subspace,

\item Frenet formulas hold:
\begin{enumerate}
\item $T^\gamma=\kappa_0 N_1^\gamma$,

\item $\widetilde{\nabla}_{N_1^\gamma}N_1^\gamma =\kappa_1 N_2^\gamma$,

\item $\widetilde{\nabla}_{N_1^\gamma}N_i^\gamma =-\kappa_{i-1} N_{i-1}^\gamma +\kappa_{i} N_{i+1}^\gamma$, $2\leq i \leq m-1$,

\item $\widetilde{\nabla}_{N_1^\gamma}N_m^\gamma =-\kappa_{m-1} N_{m-1}^\gamma$.
\end{enumerate}
\end{enumerate}

We call $(N_1^\gamma,...,N_m^\gamma)$ the {\it Frenet frame} of $\gamma$, and the functions $\kappa_i$, for $0\leq j \leq m-1$, the {\it curvatures} of $\gamma$.

If $\gamma \subset \widetilde{M}$ is parametrized by arc-length, then the curvature $\kappa_0 $ is identically equal to one. The curve $\gamma$ is called a {\it geodesic} in $\widetilde{M}$ if $\kappa_1 \equiv 0$. Notice also that if $r=1$, then $\gamma$ is said to be {\it regular}; if $r=2$, then $\gamma$ is said to {\it non-geodesic}; and if $r=3$, then $\gamma$ has nonzero second curvature.

Moreover, $\gamma$ is called a {\it circle} in $\widetilde{M}$ if it is of order $2$ and if there is a constant $R>0$ such that
\begin{eqnarray*}
\widetilde{\nabla}_{T^\gamma}T^\gamma&=&R N_2^\gamma,
\\
\widetilde{\nabla}_{T^\gamma}N_2^\gamma&=&-R T^\gamma.
\end{eqnarray*}
The constant $1/R$ is said to be {\it radius} of the circle $\gamma$ \cite{ny}.

\section{Anti-Torqued and Torqued Vector Fields} \label{tors}

Let $\widetilde{\nabla}$ be Levi-Civita connection on a Riemannian manifold $\widetilde{M}$. Denote by $\omega$ a $1$-form on $\widetilde{M}$, and by $\rho$ a smooth function on $\widetilde{M}$. The notion of a {\it torse-forming} vector field $V$ on $\widetilde{M}$ was introduced by Yano \cite{ya0} as
\begin{equation*}
\widetilde{\nabla}_X V=\rho X +  \omega(X) V, \quad \forall X\in \mathfrak{X}(\widetilde{M}).
\end{equation*}

We refer to $\omega$ and $\rho$ as the {\it generating form} and the {\it conformal scalar} (or {\it potential function}), respectively. Let $W$ be the vector field dual to $\omega$ on $\widetilde{M}$, i.e., $\langle W, X \rangle =\omega (X)$ for every $X \in \mathfrak{X}(M)$. We say that $W$ is {\it generative} of $V$ (see \cite{bo,mih0}).

A torse-forming vector field is called:
\begin{enumerate}
\item {\it Concircular} if $\omega$ is identically zero on $\widetilde{M}$ \cite{ya1},

\item {\it Recurrent} if $\rho$ is identically zero on $\widetilde{M}$ \cite{cv},

\item {\it Torqued} if $\langle V, W\rangle =0$ on $\widetilde{M}$ \cite{c01,crs3},

\item {\it Anti-torqued} if $W=-\rho V$ on $\widetilde{M}$ \cite{cs,dan,na}.
\end{enumerate}

A torqued (or anti-torqued) vector field is said to be {\it proper} if the generating form and conformal scalar are nowhere zero. This means that there is no an open subset such that $\omega$ and $\rho$ are identically zero. There are certain differences between torqued and anti-torqued vector fields on Riemannian manifolds. For example, a proper torqued vector field on a Riemannian manifold is never of constant length, whereas an anti-torqued vector field is always a unit geodesic vector field \cite{amc}. Furthermore, there is no torqued vector field that is also a gradient vector field, although there are examples of the other case \cite{crs3, dan}.

Non-trivial examples of such vector fields exist in various ambient spaces.

\begin{example}[Torqued vector fields]
Let $\widetilde{M}=I\times_\lambda M$ be a twisted product, where $I\subset \r$ is an open interval, $M$ is a Riemannian manifold, and $\lambda$ is the twisting function. Then, $\lambda\mu\frac{\partial}{\partial s}$ is a torqued vector field on $\widetilde{M}$, where $s$ is the arc-length parameter of $I$, and $\mu$ is a nonzero function on $M$ \cite{c01}.
\end{example}

\begin{example}[Anti-torqued vector fields] \label{anti-torq-example}

\begin{enumerate}
\item Let $\widetilde{M}=\s^1\times_\lambda M$ be a warped product, where $\s^1$ is the unit circle, $M$ is a Riemannian manifold, and $\lambda>0$ is the warping function on $\s^1$. Then, a globally defined unit vector field on $\s^1$ is an anti-torqued vector field \cite{dan}.

\item  Let $(x_1,...,x_m)$ be the canonical coordinates of $\r^n$. The upper half space model of the hyperbolic space is $\h^m=\{(x_1,...,x_m) \in \r^m : x_m>0\},$ and is endowed with the metric
$$
\langle ,\rangle =\frac{1}{x_m^2}\sum_{i=1}^{m}dx_i^2.
$$
In this model of $\h^m$, an orthonormal basis $\{e_1,...,e_n \}$ can be chosen as
$$
e_j=x_m\partial _j, \quad e_m=-x_m\partial _m, \quad j=1,...,m-1.
$$
The Levi-Civita connection gives, for $i=1,...,m$ and $j=1,...,m-1,$
$$
\widetilde{\nabla}_{e_j}e_{m}=e_j, \quad \widetilde{\nabla}_{e_m}e_m=0, \quad \nabla
_{e_i}e_j=0 \text{ } (i\neq j), \quad \widetilde{\nabla}_{e_j}e_j=-e_m.
$$%
It is direct to see that $e_m$ is an anti-torqued vector field, namely
$$
\widetilde{\nabla}_{X}e_m =X-\langle X,e_{m}\rangle e_m, \quad \forall X\in \mathfrak{X}(\h^m),
$$
where the conformal scalar is identically $1$ \cite{amc}.

\item Let $E$ be the radial vector field on the connected Riemannian manifold $\r^m$ punched at the origin. Then $\frac{1}{|E|}E$ is an anti-torqued vector field on $\r^m\backslash \{0\}$ \cite{dan}.
\end{enumerate}
\end{example}

\section{Anti-Torqued Slant Helices} \label{anti-helices}

Let $V$ be an anti-torqued vector field on an oriented connected $m$-dimensional Riemannian manifold $\widetilde{M}$. Then
\begin{equation*}
\widetilde{\nabla}_X V=\rho (X -\langle X, V\rangle  V), \quad \forall X\in \mathfrak{X}(\widetilde{M}),
\end{equation*}
where $\widetilde{\nabla}$ is the Levi-Civita connection, and $\rho$ is the conformal scalar.

\begin{theorem} \label{anti-slant}
Let $\widetilde{M}$ be a Riemannian manifold endowed with an anti-torqued vector field $V$, and let $\gamma \subset \widetilde{M}$ be a Frenet curve of order $r$, $1\leq r \leq m$. If $\gamma$ is an anti-torqued slant helix with axis $V$, then one the following holds:
\begin{enumerate}
\item[(a)] $\gamma$ is a geodesic in $\widetilde{M}$,

\item[(b)] $\gamma$ is a Frenet curve of order $2$, the conformal scalar of $V$ is the negative of the first curvature of $\gamma$, and $V$ is a concircular vector field along $\gamma$,

\item[(c)] $\gamma$ satisfies the system \eqref{sys-anti}.
\end{enumerate}
\end{theorem}

\begin{proof}
Let $\gamma \subset \widetilde{M}$ be a Frenet curve of order $r$, $1\leq r \leq m$, parametrized by arc-length. Denote by $(T^\gamma,N_2^\gamma,...,N_m^\gamma)$ the Frenet frame of $\gamma$, and by $\kappa_i$, for $0<j<m-1$, the curvatures of $\gamma$.

If $\gamma$ is an anti-torqued slant helix with axis $V$, then from Definition \ref{def} it follows
\begin{equation*}
\langle V , N_2^\gamma \rangle = \cos \theta , \quad \theta \in [0,2\pi].
\end{equation*}
We distinguish three cases:

{\bf Case 1.} $V$ is parallel to $T^\gamma$ along $\gamma$. Then $\theta = \frac{\pi}{2}$ (or $\theta = \frac{3\pi}{2}$) and  $V=\pm T^\gamma$. Without lose of generality, we may assume $V=T^\gamma$, meaning that $\gamma$ is the integral curve of $V$. Since $V$ is a unit geodesic vector field on $\widetilde{M}$, $\gamma$ becomes a geodesic in $\widetilde{M}$.

{\bf Case 2.} $V$ is parallel to one of $N_2^\gamma,...,N_m^\gamma$ along $\gamma$. Without lose of generality, we may assume $V=N_i^\gamma$, for $2\leq i \leq m$. Because $V$ is an anti-torqued vector field on $\widetilde{M}$, we have
$$
\widetilde{\nabla}_{T^\gamma} N_i^\gamma= \rho T^\gamma.
$$
On the other hand, the Frenet equations give
$$
\widetilde{\nabla}_{T^\gamma} N_i^\gamma=
\left\{
\begin{array}{cc}
-\kappa_1T^\gamma +\kappa_{2} N_{3}^\gamma,&  i=2,\\
-\kappa_{i-1} N_{i-1}^\gamma +\kappa_{i} N_{i+1}^\gamma,& 2<i<m, \\
-\kappa_{m-1} N_m^\gamma,& i=m.
\end{array}%
\right.
$$
Comparing these expressions, we conclude that the only possible case is $i=2$ and $\kappa_1=-\rho$. Consequently, it follows that $\theta \in \{0,\pi,2\pi\} $ and the set
$$
\{T^\gamma , \widetilde{\nabla}_{T^\gamma} T^\gamma,  \widetilde{\nabla}_{T^\gamma}(\widetilde{\nabla}_{T^\gamma} T^\gamma)  \}
$$
is linearly dependent, which means that $\gamma$ is a Frenet curve of order $2$, and the conformal scalar $\rho$ along $\gamma$ is equal to $-\kappa_1$. Moreover, $V$ is a concircular vector field along $\gamma$.

{\bf Case 3.}
$V$ is parallel to none of $T^\gamma,N_2^\gamma,...,N_m^\gamma$ along $\gamma$. Thus, $V$ admits the following orthogonal decomposition:
\begin{equation}
V= f_1^\gamma T^\gamma +\cos \theta N_2^\gamma +\sum_{i=3}^m f_i^\gamma N_i^\gamma, \label{anti-1}
\end{equation}
where $f_1^\gamma, f_3^\gamma,...,f_m^\gamma$ are smooth functions defined on $I$. Since $V$ is anti-torqued, it follows that
$$
\widetilde{\nabla}_{T^\gamma} V = \rho(T^\gamma-f_1^\gamma V),
$$
where the relation $\langle V,T^\gamma \rangle =f_1^\gamma$ has been used. Expanding this expression yields
$$
\widetilde{\nabla}_{T^\gamma} V = \rho \left ( (1-(f_1^\gamma)^2)T^\gamma -\cos \theta f_1^\gamma  N_2^\gamma -f_1^\gamma \sum_{i=3}^m f_i^\gamma N_i^\gamma\right ).
$$
Moreover, derivating \eqref{anti-1} with respect to $T^{\gamma }$ and applying the Frenet formulas, we conclude that
\begin{eqnarray*}
\widetilde{\nabla}_{T^\gamma} V &=&(T^{\gamma }(f_{1}^{\gamma })-\cos \theta \kappa_{1})T^{\gamma }+(\kappa _{1}f_{1}^{\gamma }-\kappa _{2}f_{3}^{\gamma})N_{2}^{\gamma }
\\
&&+(T^{\gamma }(f_{3}^{\gamma })+\cos \theta \kappa_{2}-\kappa _{3}f_{4}^{\gamma })N_{3}^{\gamma}
\\
&&+\sum_{i=4}^{m-1}(T^{\gamma }(f_{i}^{\gamma })+\kappa _{i-1}f_{i-1}^{\gamma}-\kappa _{i}f_{i+1}^{\gamma })N_{i}^{\gamma }
\\
&&+(T^{\gamma }(f_{m}^{\gamma})+\kappa _{m-1}f_{m-1}^{\gamma })N_{m}^{\gamma }.
\end{eqnarray*}
By equating the tangent and normal components in the last two relations, we obtain
\begin{equation}
\left.
\begin{array}{l}
T^{\gamma }(f_{1}^{\gamma })-\cos \theta \kappa_{1} -\rho  ( 1-(f_1^\gamma)^2)=0, \\
\kappa _{1}f_{1}^{\gamma }-\kappa _{2}f_{3}^{\gamma}+\cos \theta \rho f_1^\gamma =0, \\
T^{\gamma }(f_{3}^{\gamma })+\cos \theta \kappa_{2}-\kappa _{3}f_{4}^{\gamma }+\rho f_1^\gamma f_3^\gamma=0,\\
T^{\gamma }(f_{i}^{\gamma })+\kappa _{i-1}f_{i-1}^{\gamma}-\kappa _{i}f_{i+1}^{\gamma }+\rho f_1^\gamma f_i^\gamma=0, \quad 4\leq i \leq m-1,\\
T^{\gamma }(f_{m}^{\gamma})+\kappa _{m-1}f_{m-1}^{\gamma }+\rho f_1^\gamma f_m^\gamma=0.
\end{array}%
\right.  \label{sys-anti}
\end{equation}
\end{proof}

We now consider the particular case $m=3$. Let $\gamma \subset \widetilde{M}$ be a Frenet curve of order $r$, $1\leq r \leq 3$. In this case, we can adopt the usual notations $\kappa_{1}=\kappa$, $\kappa_{2}=\tau$, and
$$
N_2^{\gamma }=N^{\gamma }, \quad N_3^{\gamma }=B^{\gamma }, \quad f_{1}^{\gamma }=f^{\gamma }, \quad f_{3}^{\gamma }=g^{\gamma }.
$$
Suppose that $\gamma$ is an anti-torqued slant helix with axis $V$. Then, the system \eqref{sys-anti} is given by
\begin{equation}
\left.
\begin{array}{l}
T^{\gamma }(f^{\gamma })-\cos \theta \kappa -\rho  ( 1-(f^\gamma)^2)=0, \\
\kappa f^{\gamma }-\tau g^{\gamma}+\cos \theta \rho f^\gamma =0, \\
T^{\gamma }(g^{\gamma })+\cos \theta \tau+\rho f^\gamma g^\gamma=0.
\end{array}%
\right.  \label{sys-anti-3}
\end{equation}

In the particular case $\theta=\frac{\pi}{2}$ ($\theta=\frac{3\pi}{2}$), we obtain the following result.

\begin{theorem} \label{anti-slant-3d}
Let $\widetilde{M}$ be a $3$-dimensional Riemannian manifold endowed with an anti-torqued vector field $V$, and let $\gamma \subset \widetilde{M}$ be a Frenet curve of order $r$, $1\leq r \leq 3$. If $\gamma$ is an anti-torqued slant helix with axis $V$ and $\theta=\frac{\pi}{2}$ (or $\theta=\frac{3\pi}{2}$), then either $\gamma$ is a geodesic in $\widetilde{M}$ or the ratio of its curvatures is given by
$$
 \frac{\tau}{\kappa}=c\frac{f^\gamma}{\sqrt{1-(f^\gamma)^2}},
$$
where $f^\gamma$ is tangential component of $V$ along $\gamma$, and $c>0$ is a constant.
\end{theorem}
\begin{proof}
Since $\theta=\frac{\pi}{2}$ (or $\theta=\frac{3\pi}{2}$), the system \eqref{sys-anti-3} can be simplified by
\begin{equation}
\left.
\begin{array}{l}
T^{\gamma }(f^{\gamma })=\rho  ( 1-(f^\gamma)^2), \\
\kappa f^{\gamma }=\tau g^{\gamma}, \\
T^{\gamma }(g^{\gamma })=-\rho f^\gamma g^\gamma,
\end{array}%
\right. \label{sys-3d}
\end{equation}
where $f^{\gamma }$ is nonzero function on $I$ because $\rho$ is nowhere zero on $\widetilde{M}$. If $\tau \equiv 0$, then the second equality in \eqref{sys-3d} implies that $\gamma$ is a geodesic in $\widetilde{M}$. From the first and third equations of \eqref{sys-3d}, we get
$$
\frac{T^\gamma (g^\gamma)}{g^\gamma}=-\frac{f^\gamma T^\gamma (f^\gamma)}{1-(f^\gamma)^2}.
$$
A first integration yields
$$
g^{\gamma }=c\sqrt{1-(f^\gamma)^2},
$$
where $c>0$ is a constant. The proof is completed by substituting this expression into the second equality in \eqref{sys-3d}.
\end{proof}

\begin{remark}
Assume that $V$ is a unit parallel transported vector field along a non-geodesic curve $\gamma$ (i.e. the case where $\rho=0$ along $\gamma$) and that $\langle V, T^\gamma \rangle = \cos \theta$, $\theta \in [0,2\pi]$. This condition is equivalent to the relation $\langle V, N^\gamma \rangle =0$. However, if $V$ is an anti-torqued vector field, then the condition $\langle V, T^\gamma \rangle = \cos \theta$ does not imply $\langle V, N^\gamma \rangle =0$. Therefore, Theorem \ref{anti-slant-3d} generalizes the results established in \cite[Theorem A, Theorem 1]{dfds}. More clearly, in the system \eqref{sys-3d}, if $\rho\equiv 0$, then both $f^\gamma$ and $g^\gamma$ become constant, implying that $\frac{\tau}{\kappa}$ is a constant. This means that $\gamma$ is a parallel general helix (see \cite{dfds}) .
\end{remark}

We now present some examples to illustrate these results.

\subsection*{Examples of anti-torqued slant helices}
Let $E $ be the radial vector in the connected Riemannian manifold $\widetilde{M}=\r^m\backslash \{0\}$. Then%
$$
V=\frac{E}{|E| }
$$%
is an anti-torqued vector field defined globally on $\widetilde{M}$, whose conformal scalar is $\rho=\frac{1}{|E|} $. Let $\gamma (s)\subset \widetilde{M}$, $s \in I \subset \r$, be a Frenet curve of order $r$, $1\leq r \leq m$, parametrized by arc-length. Assume that $\gamma (s)$ is a anti-torqued slant helix with axis $V$. Since $V|_{\gamma(s)} =\frac{\gamma(s)}{|\gamma(s)|}$, by Definition \ref{def} we have
\begin{equation}
\langle \gamma(s), N^\gamma  \rangle =\cos \theta |\gamma(s)|, \quad \theta \in [0,2\pi]. \label{ex-41}
\end{equation}

Then we can give the following examples of anti-torqued slant helices in the Euclidean setting.

\begin{example}
Assume that $\theta=\frac{\pi}{2}$ (or $\theta=\frac{3\pi}{2}$), and that $\widetilde{M}=\r^3\backslash \{0\}$. If $\gamma$ is a straight line passing through the origin, then the expression \eqref{ex-41} obviously holds. Moreover, when $r=3$, the expression \eqref{ex-41} characterizes rectifying curves in $\widetilde{M}$. It is known from \cite{crc0} that
$$
\gamma(s)=(s+b)T^\gamma +a B^\gamma, \quad a,b\in \r, \quad a\neq 0,
$$
and consequently,
$$
V|_{\gamma(s)} =f^\gamma(s)T^\gamma (s) + g^\gamma(s) B^\gamma(s),
$$
where
$$
f^\gamma(s)=\frac{s+b}{\sqrt{a^2+(s+b)^2}}, \quad  g^\gamma(s)=\frac{a}{\sqrt{a^2+(s+b)^2}}.
$$
The ratio of curvatures of $\gamma(s)$ is $(s+b)/a$ (see \cite[Theorem 2]{crc0}); this confirms the result in Theorem \ref{anti-slant-3d}.
\end{example}

\begin{example} \label{exm-zero-tors}
Let $\widetilde{M}=\r^3\backslash \{0\}$, and let $(x,y,z)$ denote the Cartesian coordinates. Consider the logarithmic spiral lying in the $xz$-plane (see Figure \ref{fig1})
$$
\gamma (s)=\frac{s}{\sqrt{2}}\left( \cos \left( \log s\right) ,0,\sin \left( \log s\right) \right) , \quad s>0,
$$%
with curvature $\kappa(s)=\frac{1}{s}$. The conformal scalar of the anti-torqued vector field $V=\frac{E}{|E| } $ along $\gamma$ is
$$
\rho|_{\gamma(s)}=\frac{1}{|\gamma(s)|}=\frac{\sqrt{2}}{s}, \quad s>0,
$$
where $E$ is the radial vector field in $\r^3$. A direct calculation shows that the angle between $V$ and $N^\gamma$ along $\gamma$ is $\theta =\frac{3\pi}{4}$ (or $\theta =\frac{7\pi}{4}$), and that
$$
f^\gamma(s)=\frac{1}{\sqrt{2}}, \quad g^\gamma(s)=0.
$$
It is easy to verify that the system \eqref{sys-anti-3} is satisfied.
\begin{figure}[hbtp]
\begin{center}
\includegraphics[width=.3\textwidth]{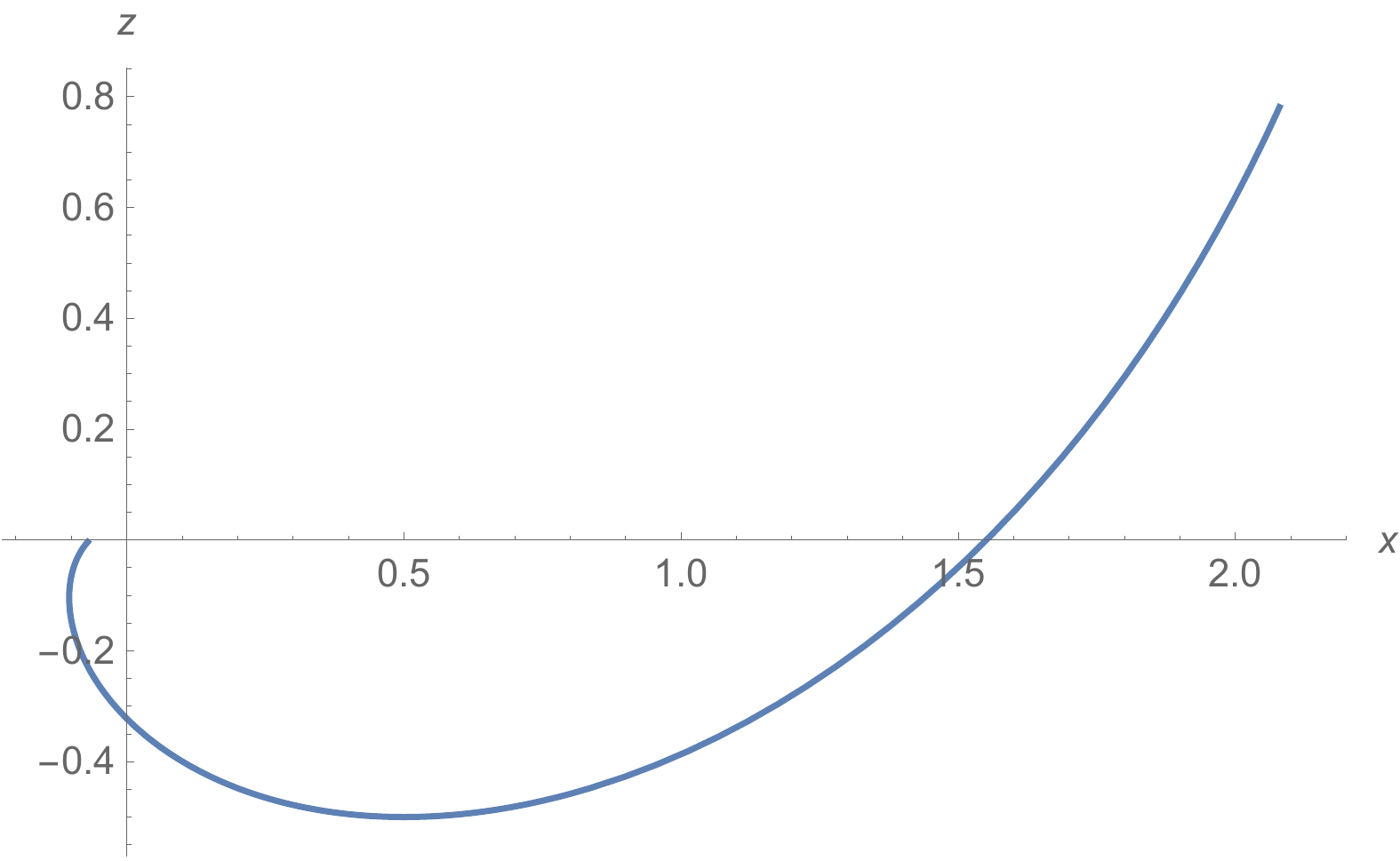} \includegraphics[width=.2\textwidth]{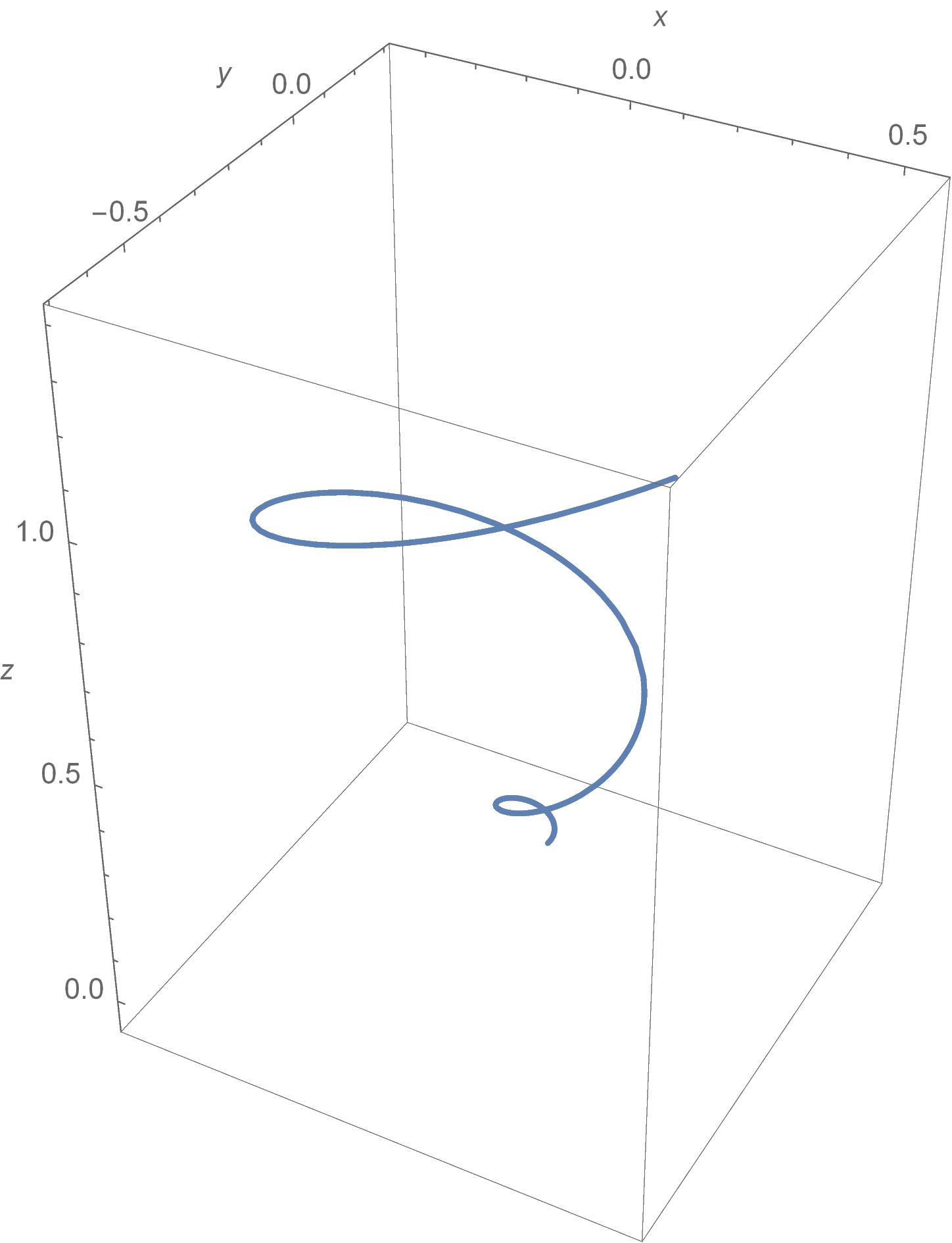}
\end{center}
\caption{Left: Logaritmic spiral as an anti-torqued slant helix. Right: Loxodrome as an anti-torqued slant helix.}\label{fig1}
\end{figure}
\end{example}


\begin{example} \label{exm-nozero-tors}
Let $\widetilde{M}=\r^3\backslash \{0\}$. Consider the loxodrome of a circular cone (see Figure \ref{fig1}), which is also a general helix, given by
$$
\gamma (s)=\frac{s}{4}\left( \cos m(s), \sqrt{3},\sin m(s) \right) , \quad s>0,
$$
where $m(s)=2\sqrt{3}\log \left (\frac{s}{4} \right )$ (see \cite{adb}). The conformal scalar of the anti-torqued vector field $V=\frac{E}{|E| } $ along $\gamma$ is
$$
\rho|_{\gamma(s)}=\frac{1}{|\gamma(s)|}=\frac{1}{s}, \quad s>0,
$$
where $E$ is the radial vector field in $\r^3$. The vector fields $V$ and $N^{\gamma}$ along $\gamma$ intersect at an angle $\theta =\arccos \left ( \frac{-\sqrt{3}}{\sqrt{13}}\right )$. The curvatures of $\gamma$ are given by
$$
\kappa(s)=\frac{\sqrt{39}}{2s}, \quad \tau(s)=-\frac{3}{2s}.
$$
We also compute
$$
f^\gamma=\frac{1}{2}, \quad g^\gamma=\frac{7\sqrt{3}}{4\sqrt{13}}.
$$
We can see that \eqref{sys-anti-3} is satisfied.
\end{example}

After presenting Examples \ref{exm-zero-tors} and \ref{exm-nozero-tors}, we are able to classify anti-torqued slant helices in $\r^3\backslash\{0\}$ with axis $\frac{E}{|E|}$, where $E$ is the radial vector field in $\r^3\backslash\{0\}$.

\begin{theorem} \label{anti-slant2}
Let $E$ be the radial vector field in $\widetilde{M}=\r^3\backslash\{0\}$, and let $\gamma \subset\widetilde{M}$ be a Frenet curve of order $r$, $1\leq r \leq 3$. If $\gamma$ is an anti-torqued slant helix with axis $\frac{E}{|E|}$, then one the following holds:
\begin{enumerate}
\item[(a)]  $\gamma$ is a circle with centered at the origin, lying an affine plane of
$\widetilde{M}$ that passes through the origin,

\item[(b)]  $\gamma$ is a logaritmic spiral lying an affine plane of $\widetilde{M}$,

\item[(c)]  $\gamma$ is a rectifying curve in $\widetilde{M}$,

\item[(d)]  $\gamma$ satisfies Equations \eqref{kappa} and \eqref{tau}.
\end{enumerate}
\end{theorem}

\begin{proof}
Let $V=\frac{E}{|E|}$ be an anti-torqued vector field on $\widetilde{M}=\r^3\backslash\{0\}$, with conformal scalar $\rho=\frac{1}{|E|}$. Assume that $\gamma=\gamma(s)$, $s\in I$, is an anti-torqued slant helix with axis $V$, parametrized by arc-length. Since  $V|_\gamma=\frac{\gamma}{|\gamma|}$, we have
$$
\langle \gamma(s), N^\gamma(s)\rangle=\phi(s)\cos \theta,
$$
where $N^\gamma$ is the principal normal vector field of $\gamma$, and $\phi(s)=|\gamma(s)|$ is the distance function of $\gamma(s)$. If $\theta=\frac{\pi}{2}$ (or $\theta=\frac{3\pi}{2}$), then the above expression characterizes rectifying curves. We will assume that $\theta\notin \{\frac{\pi}{2},\frac{3\pi}{2}\}$. Derivating $\phi(s)$, we obtain
$$
\phi (s)\phi '(s)=\langle \gamma (s),T^{\gamma}(s)\rangle.
$$
Thus, the following orthogonal decomposition of $\gamma(s)$ holds:
\begin{equation}
\gamma (s)=\phi (s)\phi'(s)T^{\gamma }(s)+\cos \theta \phi
(s)N^{\gamma }(s)+g(s)B^{\gamma }(s),  \label{frame}
\end{equation}%
where $g(s)$ is a smooth function on the interval $I$. Derivating Equation \eqref{frame}, we conclude that
\begin{equation}
\left.
\begin{array}{r}
\left( \phi (s)\phi'(s)\right)'-\cos \theta \phi (s)\kappa
(s)-1=0, \\
\left( \cos \theta +\kappa (s)\phi (s)\right) \phi'(s)-g(s)\tau (s)=0,
\\
g^{\prime }(s)+\cos \theta \phi (s)\tau (s)=0.%
\end{array}%
\right.  \label{sys}
\end{equation}
To solve the system \eqref{sys}, we consider two cases:

{\bf Case:} Assume that $\tau (s)=0,$ for every $s \in I$. In this case, \eqref{frame} and \eqref{sys} reduce, respectively, to
\begin{equation}
\gamma (s)=\phi (s)\phi ^{\prime }(s)T^{\gamma }(s)+\cos \theta  \phi
(s)N^{\gamma }(s)  \label{TN}
\end{equation}%
and%
\begin{equation}
\left.
\begin{array}{r}
\left( \phi (s)\phi ^{\prime }(s)\right) ^{\prime }-\cos \theta  \kappa (s)\phi
(s)-1=0, \\
\left( \cos \theta +\kappa (s)\phi (s)\right) \phi ^{\prime }(s)=0.%
\end{array}%
\right.  \label{sys-tau=0}
\end{equation}
If $\phi (s)$ is a positive constant, which is equivalent to saying that $\gamma (s)$ is a circle centered at the origin, then from Equation \eqref{TN}, we have $\cos \theta =\pm 1$.  However, from the first equality of the system \eqref{sys-tau=0}, we deduce $\cos \theta =- 1$. If $\phi' (s)\neq 0$ for every $s,$ then $\cos \theta =-\kappa (s)\phi (s)$. On the other hand, because $\left\langle \gamma (s),\gamma (s)\right\rangle =\phi (s)^{2},$ from
Equation \eqref{TN} we have $\phi'(s)=\pm \sin \theta. $ Integrating this expression, we obtain
$$
\phi (s)=\pm \sin \theta s+c,
$$
for a constant $c.$ Moreover,
$$
\kappa (s)=-\frac{\cos \theta }{\pm \sin \theta s+c},
$$
which corresponds to a logarithmic spiral \cite[p. 138]{gr}.

{\bf Case:} Suppose that $\tau (s)\neq 0$ for every $s \in I.$ From the system \eqref{sys}, we conclude that
\begin{equation}
\kappa (s)=\frac{\left( \phi (s)\phi ^{\prime }(s)\right) ^{\prime }-1}{%
\cos \theta \phi (s)}  \label{kappa}
\end{equation}%
and
\begin{equation}
\tau ^{\prime }(s)=\frac{\tau (s)}{F(s)}\left( F^{\prime }(s)+(\cos \theta ) ^{2}\phi
(s)\tau (s)^{2}\right) ,   \label{tau}
\end{equation}%
where $F(s)=\left( (\cos \theta ) ^{2}+\left( \phi (s)\phi'(s)\right)
^{\prime }-1\right) \phi'(s).$
\end{proof}

The existence of anti-torqued slant helix can also be observed in other ambient spaces, as described below.

\begin{example}
Consider $\gamma$ as a vertical straight line in $\h^m$, which is a geodesic. Then, $\gamma$ is an anti-torqued slant helix with axis $e_m$, where $T^\gamma$ is parallel to the anti-torqued vector field $e_m$ (see also Example \ref{anti-torq-example}).
\end{example}

\begin{example}
Let $(\widetilde{M},\phi ,\xi ,\eta ,g)$ be a $\beta$-Kenmotsu manifold. The vector field $\xi$ is an anti-torqued vector field \cite{dan}. It is well-known that if the dimension of $\widetilde{M}$ is $3$, then for every Legendre curve the unit normal vector field is parallel to $\xi$ (see \cite{ccm}). This means that every Legendre curve is an anti-torqued slant helix with axis $\xi$.
\end{example}

\section{Torqued Curves} \label{torq-curv}

Recall that a torqued vector field $V$ on $\widetilde{M}$ satisfies
\begin{equation*}
\widetilde{\nabla}_X V=\rho X +\omega V, \quad \forall X\in \mathfrak{X}(\widetilde{M}),
\end{equation*}
where $\omega(V)=0$. Denoting by $W$ the dual vector field to $\omega$, we conclude that $\langle V, W \rangle =0$ on $\widetilde{M}$. In particular, when $\omega\equiv 0$, the vector field $V$ reduces to a concircular vector field.

Let $\gamma=\gamma(s)$, $s\in I$, be a Frenet curve in $\widetilde{M}$ of order $1\leq r \leq m$, parametrized by arc-length. As in previous sections, we denote by $(T^\gamma,N_2^\gamma,...,N_m^\gamma)$ the Frenet frame of $\gamma$, and by $\kappa_i$, for $0<j<m-1$, the curvatures of $\gamma$.

We call $\gamma(s)$ a {\it torqued curve} with respect to $V$ if $\langle V|_{\gamma(s)}, N_2^\gamma(s) \rangle =\theta$, with $\theta \in \r$ constant, for every $s\in I$.

\begin{theorem} \label{tor-slant}
Let $\widetilde{M}$ be a Riemannian manifold endowed with a torqued vector field $V$, and let $\gamma \subset \widetilde{M}$ be a Frenet curve of order $r$, $1\leq r \leq m$. If $\gamma$ is a torqued curve with respect to $V$, then one the following holds:
\begin{enumerate}
\item[(a)] $\gamma$ is a geodesic in $\widetilde{M}$, and $V$ is a concircular vector field along $\gamma$,

\item[(b)] $\gamma$ is a Frenet curve of order $2$, the conformal scalar of $V$ is a multiple of the first curvature of $\gamma$, and $V$ is a concircular vector field along $\gamma$,

\item[(c)] $\gamma$ satisfies the system \eqref{sys-tor}.
\end{enumerate}
\end{theorem}

\begin{proof}
If $\gamma(s)$, $s\in I$, is a torqued curve with respect to $V$, then it follows that
\begin{equation*}
\langle V , N_2^\gamma \rangle = \theta , \quad \theta \in \r.
\end{equation*}
We consider three cases:

{\bf Case 1.} $V$ is parallel to $T^\gamma$ along $\gamma$. Then $\theta =0 $ and  $V=f^\gamma T^\gamma$ with $f^\gamma$ a smooth function on $I$. Since $\omega(T)=0$, we have $\widetilde{\nabla}_{T^\gamma}V=\rho T^\gamma$. On the other hand, differentiating $V=f^\gamma T^\gamma$ along $\gamma$ gives
$$
\widetilde{\nabla}_{T^\gamma}V=T^\gamma(f^\gamma)T^\gamma+f^\gamma \kappa_1 N_2^\gamma.
$$
Comparing these expressions, we deduce that $T^\gamma(f^\gamma)=\rho$ and $\kappa_1\equiv 0$, because $V$ is nowhere zero on $\widetilde{M}$.

{\bf Case 2.} $V$ is parallel to one of $N_2^\gamma,...,N_m^\gamma$ along $\gamma$. Without lose of generality, we may assume $V=f_i^\gamma N_i^\gamma$, $2\leq i \leq m$, for a smooth function $f_i^\gamma$ on $I$. Because $\omega(N_i^\gamma)=0$, we have
$$
\widetilde{\nabla}_{T^\gamma} (f_i^\gamma N_i^\gamma)= \rho T^\gamma.
$$
On the other hand, the Frenet equations imply that
$$
\widetilde{\nabla}_{T^\gamma} (f_i^\gamma N_i^\gamma)=
\left\{
\begin{array}{cc}
T^\gamma(f_2^\gamma)N_2^\gamma-f_2^\gamma\kappa_1T^\gamma +f_2^\gamma\kappa_{2} N_{3}^\gamma,&  i=2,\\
T^\gamma(f_i^\gamma)N_i^\gamma-f_i^\gamma\kappa_{i-1} N_{i-1}^\gamma +f_i^\gamma\kappa_{i} N_{i+1}^\gamma,& 2<i<m, \\
T^\gamma(f_m^\gamma)N_i^\gamma-f_m^\gamma\kappa_{m-1} N_m^\gamma,& i=m.
\end{array}%
\right.
$$
Similar to the proof of Theorem \ref{anti-slant}, comparing these expressions, we conclude that it must be $i=2$, $\kappa_1f_2^\gamma=-\rho$, $\kappa_2\equiv0$, and $T^\gamma(f_2^\gamma)=0$. Moreover, $V$ is a concircular vector field along $\gamma$.

{\bf Case 3.}
$V$ is parallel to none of $T^\gamma,N_2^\gamma,...,N_m^\gamma$ along $\gamma$. Let $W$ be generative of $V$. Thus, there exist smooth functions $f_1^\gamma,f_3^\gamma,...,f_m^\gamma$ and $g_1^\gamma,g_2^\gamma,...,g_m^\gamma$ defined on $I$ such that
\begin{equation}
V= f_1^\gamma T^\gamma + \theta N_2^\gamma +\sum_{i=3}^m f_i^\gamma N_i^\gamma, \label{torq-1}
\end{equation}
and
\begin{equation}
W= g_1^\gamma T^\gamma + \sum_{i=2}^m g_i^\gamma N_i^\gamma. \label{torq-2}
\end{equation}
Since $V$ is a torqued vector field, we have
$$
\widetilde{\nabla}_{T^\gamma} V = \rho(T^\gamma+g_1^\gamma V).
$$
Expanding this expression yields
$$
\widetilde{\nabla}_{T^\gamma} V = \rho \left ( (1+f_1^\gamma g_1^\gamma)T^\gamma + \theta  g_1^\gamma  N_2^\gamma +g_1^\gamma \sum_{i=3}^m f_i^\gamma N_i^\gamma\right ).
$$
Derivating \eqref{torq-1} with respect to $T^{\gamma }$, we get
\begin{eqnarray*}
\widetilde{\nabla}_{T^\gamma} V &=&(T^{\gamma }(f_{1}^{\gamma })- \theta \kappa_{1})T^{\gamma }+(\kappa _{1}f_{1}^{\gamma }-\kappa _{2}f_{3}^{\gamma})N_{2}^{\gamma }
\\
&&+(T^{\gamma }(f_{3}^{\gamma })+ \theta \kappa_{2}-\kappa _{3}f_{4}^{\gamma })N_{3}^{\gamma}
\\
&&+\sum_{i=4}^{m-1}(T^{\gamma }(f_{i}^{\gamma })+\kappa _{i-1}f_{i-1}^{\gamma}-\kappa _{i}f_{i+1}^{\gamma })N_{i}^{\gamma }
\\
&&+(T^{\gamma }(f_{m}^{\gamma})+\kappa _{m-1}f_{m-1}^{\gamma })N_{m}^{\gamma }.
\end{eqnarray*}
From the last two relations and $\langle V,W \rangle =0$, we obtain
\begin{equation}
\left.
\begin{array}{l}
T^{\gamma }(f_{1}^{\gamma })- \theta \kappa_{1} -\rho  ( 1+f_1^\gamma g_1^\gamma)=0, \\
\kappa _{1}f_{1}^{\gamma }-\kappa _{2}f_{3}^{\gamma}+\theta \rho g_1^\gamma =0, \\
T^{\gamma }(f_{3}^{\gamma })+\theta \kappa_{2}-\kappa _{3}f_{4}^{\gamma }+\rho f_3^\gamma g_1^\gamma=0,\\
T^{\gamma }(f_{i}^{\gamma })+\kappa _{i-1}f_{i-1}^{\gamma}-\kappa _{i}f_{i+1}^{\gamma }+\rho  f_i^\gamma g_1^\gamma=0, \quad 4\leq i \leq m-1,\\
T^{\gamma }(f_{m}^{\gamma})+\kappa _{m-1}f_{m-1}^{\gamma }+\rho  f_m^\gamma g_1^\gamma=0, \\
f_1^\gamma g_1^\gamma +\theta g_2^\gamma +\sum_{i=3}^m f_i^\gamma g_i^\gamma =0.
\end{array}%
\right.  \label{sys-tor}
\end{equation}
\end{proof}

As in Section \ref{anti-helices}, we consider the particular case $m=3$ and adopt the usual notations. If $\gamma$ is a torqued curve with respect to $V$, then the system \eqref{sys-tor} reduces to
\begin{equation}
\left.
\begin{array}{l}
T^{\gamma }(f^{\gamma }_1)- \theta \kappa -\rho  ( 1+f_1^\gamma g_1^\gamma)=0, \\
\kappa f_1^\gamma-\tau f_3^\gamma+ \theta \rho g_1^\gamma =0, \\
T^{\gamma }(f_3^\gamma)+\theta \tau+\rho f_3^\gamma g_1^\gamma =0, \\
f_1^\gamma g_1^\gamma +\theta g_2^\gamma +f_3^\gamma g_3^\gamma =0.
\end{array}%
\right.  \label{sys-tor-3}
\end{equation}

Suppose that $V$ is a concircular vector field on $\widetilde{M}$. This is equivalent to assuming that $g_i^\gamma \equiv 0$, $i=1,2,3$, in the system \eqref{sys-tor-3}. Consequently, we have
\begin{equation}
\left.
\begin{array}{l}
T^{\gamma }(f^{\gamma }_1)= \theta \kappa +\rho, \\
\kappa f_1^\gamma=\tau f_3^\gamma, \\
T^{\gamma }(f_3^\gamma)=-\theta \tau,
\end{array}%
\right.  \label{sys-tor-31}
\end{equation}
where $f_1^\gamma f_3^\gamma $ is nowhere zero in $I$. From the second equality in \eqref{sys-tor-31}, we write
$$
\frac{f^{\gamma }_1}{f^{\gamma }_3}=\frac{\tau}{\kappa}.
$$
Set $\phi:=\frac{\tau}{\kappa}$. Derivating this expression and using the first and third equalities in \eqref{sys-tor-31}, we obtain
\begin{equation}
T^{\gamma }(\phi )=\frac{1}{f^{\gamma }_3}(\theta\kappa (1 +\phi^2)+\rho ).
\label{conc-0}
\end{equation}
If $T^{\gamma }(\phi )$ is nowhere zero in $I$, then by the third equality in \eqref{sys-tor-31}, we obtain
\begin{equation}
\theta T^{\gamma }\left (\frac{\kappa (1 +\phi^2)}{T^{\gamma }(\phi )} \right )+T^{\gamma }\left (\frac{\rho}{T^{\gamma }(\phi )} \right )+\theta \kappa\phi =0. \label{conc-1}
\end{equation}
This ODE characterizes concircular curves in arbitrary Riemannian manifolds, and thus, we can state the following result.
\begin{theorem}\label{conc}
Let $\widetilde{M}$ be a Riemannian manifold endowed with a concircular vector field $V$, and let $\gamma \subset \widetilde{M}$ be a Frenet curve of order $3$. If $\gamma$ is a concircular curve with respect to $V$ and $T^{\gamma }(\phi )\neq 0$, then $\gamma$ satisfies \eqref{conc-1}.
\end{theorem}

Remark that if $\widetilde{M}=\r^3$, then the concircular factor $\rho$ of $V$ is a constant. Consequently, Theorem \ref{conc} reduces to \cite[Theorem 2.4]{loy4}.

\section*{Acknowledgments}
This study was supported by Scientific and Technological Research Council of Turkey (TUBITAK) under the Grant Number (123F451). The authors thank to TUBITAK for their supports.

\section*{Declarations}

\begin{itemize}
\item Conflict of interest: The authors declare no conflict of interest.

\item Ethics approval: Not applicable.

\item Availability of data and materials: Not applicable.

\item Code availability: Not applicable.

\item Authors’ contributions: Each author contributes equally to the study.
\end{itemize}


\medskip


\begin{thebibliography}{99}

\bibitem{adb} F. K. Aksoyak, B. B. Demirci, M. Babaarslan, Characterizations of loxodromes on rotational surfaces in Euclidean 3-space, Int. Electron. J. Geom. 16 (2023), 147–159.

\bibitem{agy} O. Ateş, İ. Gök, Y. Yaylı, A new representation for slant curves in Sasakian $3$-manifolds, Int. Electron. J. Geom. 17 (2024), 277–289.

\bibitem{amc} M. E. Aydın, A. Mihai, C. Özgür, Torqued and anti-torqued vector fields on hyperbolic spaces, submitted.

\bibitem{amc1} M. E. Aydın, A. Mihai, C. Özgür, Rectifying submanifolds of Riemannian manifolds with anti-torqued axis, to appear in J. Korean Math. Soc.

\bibitem{ahy} M. E. Aydın, A. Has, B. Yılmaz, Multiplicative rectifying submanifolds of multiplicative Euclidean space, Math. Meth. Appl. Sci. 48 (2025), 329–339.

\bibitem{bar} M. Barros, General helices and a theorem of Lancret, Proc. Am. Math. Soc. 125 (1997), 1503-1509.

\bibitem{bo} A. M. Blaga, C. Özgür, Almost $\eta$-Ricci and almost $\eta$-Yamabe solitons with torse-forming potential vector field, Quaest. Math. 45 (2022), 143-163.

\bibitem{bo1} A. M. Blaga, C. Özgür, On torse-forming-like vector fields, Mediterr. J. Math. 21 (2024), Paper No. 214, 19 pp.

\bibitem{bm} A. Burlacu, A. Mihai, Applications of differential geometry of curves in roads design, Romanian Journal of Transport Infrastructure (12)
2023, 1-13.

\bibitem{cgb} S. Cambie, W. Goemans, and I. V. Bussche, Rectifying curves in $n$-dimensional Euclidean space, Turk. J. Math. 40 (2016), 210--223.

\bibitem{ccm} C. Călin, M. Crasmareanu, M. I. Munteanu, Slant curves in three-dimensional $f$-Kenmotsu manifolds, J. Math. Anal. Appl. 394 (2012), 400--407.

\bibitem{cs} S. K. Chaubey, Y. J. Suh, Riemannian concircular structure manifolds, Filomat 36 (2022), 6699--6711.

\bibitem{crc0} B.-Y. Chen, When does the position vector of a space curve always lie in its rectifying plane?, Amer. Math. Monthly 110 (2003), 147--152.

\bibitem{crc1} B.-Y. Chen, F. Dillen, Rectifying curves as centrodes and extremal curves, Bull. Inst. Math. Acad. Sinica 33 (2005), 77--90.

\bibitem{crc2} B.-Y. Chen, Rectifying curves and geodesics on a cone in the Euclidean 3-space, Tamkang J. Math. 48 (2017), 209--214.

\bibitem{crs0} B.-Y. Chen, Differential geometry of rectifying submanifolds, Int. Electron. J. Geom.  9 (2016), 1--8.

\bibitem{crs1} B.-Y. Chen, Addentum to: Differential geometry of rectifying submanifolds, Int. Electron. J. Geom. 10 (2017), 81–82.

\bibitem{crs2} B.-Y. Chen, Y. M. Oh, Classification of rectifying space-like submanifolds in pseudo-Euclidean spaces, Int. Electron. J. Geom. 10 (2017), 86--95.

\bibitem{crs3} B.-Y. Chen, Rectifying submanifolds of Riemannian manifolds and torqued vector fields, Kragujevac J. Math. 41 (2017), 93--103.

\bibitem{cbook} B.-Y. Chen, Differential Geometry of Warped Product Manifolds and Submanifolds, World Scientific Publishing Co. Pte. Ltd., Hackensack, NJ, 2017.

\bibitem{c0} B.-Y. Chen, A simple characterization of generalized Robertson-Walker space-times, Gen. Relativity Gravitation 46 (2014), 1-5.

\bibitem{c00} B.-Y. Chen, L. Verstraelen, A link between torse-forming vector fields and rotational hypersurfaces, Int. J. Geom. Methods Mod. Phys. 14 (2017), 1750177, 10 pp.

\bibitem{c01} B.-Y. Chen, Classification of torqued vector fields and its applications to Ricci solitons, Kragujevac J. Math. 41 (2017), 239-250.

\bibitem{cv} C. D. Collinson, E. G. L. R. Vaz, Killing pairs constructed from a recurrent vector field. Gen. Relativity Gravitation 27 (1995), 751-759.

\bibitem{dfds} L. C. B. da Silva, G. S. Ferreira Jr., J. D. da Silva, Curves and surfaces making a constant angle with a parallel transported direction in Riemannian spaces, arXiv:2403.10716 (2024).

\bibitem{casn} A. \c{C}al\i \c{s}kan, B. \c{S}ahin, Slant helices on Riemannian manifolds, Filomat 38 (2024), 7743-7754.


\bibitem{dca} S. Deshmukh, B.-Y. Chen, S. H. Alshammari, On rectifying curves in Euclidean 3-space, Turk. J. Math. 42 (2018), 609-620.

\bibitem{dan} S. Deshmukh, I. Al-Dayel, D. M. Naik, On an anti-torqued vector field on Riemannian manifolds, Mathematics 9 (2021), 2201.

\bibitem{dta} S. Deshmukh, N. B. Turki, H. Alodan, On the differential equation governing torqued vector fields on a Riemannian manifold, Symmetry 12 (2020), 1941.

\bibitem{dmtv} S. Deshmukh, J. Mikes, N. B. Turki, G.-E. Vîlcu, A note on geodesic vector fields, Mathematics 8 (2020), 1663.

\bibitem{dr} K.L. Duggal, R. Sharma,  Symmetries of spacetimes and Riemannian manifolds, Mathematics and its Applications 487, Kluwer Academic Publishers, Dordrecht, 1999.

\bibitem{gr} A. Gray, Modern Differential Geometry of Curves and Surfaces with Mathematica. Second edition, CRC Press, Boca Raton, FL, 1998.

\bibitem{inpt} K. İlarslan, E. Ne\v{s}ovic, M. Petrovi\'{c}-Torgasev, Some characterizations of rectifying curves in the Minkowski 3-space, Novi Sad J. Math. 33 (2003), 23-32.

\bibitem{in} K. İlarslan, E. Ne\v{s}ovic, On rectifying curves as centrodes and extremal curves in the Minkowski 3-space, Novi Sad J. Math. 37 (2007), 53-64.

\bibitem{id}  A. Ishan, S. Deshmukh, Torse-forming vector fields on $m-$spheres, AIMS Math. 7 (2022), 3056-3066.

\bibitem{it}  S. Izumiya, N. Takeuchi, New special curves and developable surfaces, Turk. J. Math. 28 (2004), 153-163.

\bibitem{jamr} M. Jianu, S. Achimescu, L. Daus, A. Mihai, O-A. Roman, D. Tudor, Characterization of rectifying curves by their involutes and evolutes, Mathematics 9 (2021), 3077.

\bibitem{lmm} M. C. López, V. F. Mateos, J. M. Masqué, The equivalence problem of curves in a Riemannian manifold, Ann. Mat. Pura Appl. 194 (2015), 343-367.

\bibitem{loy0} P. Lucas, J. A. Ortega-Yag\"ues, Rectifying curves in the three-dimensional sphere, J. Math. Anal. Appl. 421 (2015), 1855-1868.

\bibitem{loy1} P. Lucas, J. A. Ortega-Yag\"ues, Rectifying curves in the three-dimensional hyperbolic space, Mediterr. J. Math. 13 (2016), 2199-2214.

\bibitem{loy2} P. Lucas, J. A. Ortega-Yag\"ues, Slant helices in the three-dimensional sphere, J. Korean Math. Soc. 54 (2017) 1331-1343.

\bibitem{loy3} P. Lucas, J. A. Ortega-Yag\"ues, Helix surfaces and slant helices in the three-dimensional anti-De Sitter space, Rev. R. Acad. Cienc. Exactas F\'{\i}s. Nat. Ser. A Mat. RACSAM 111 (2017), 1201-1222.

\bibitem{loy4} P. Lucas, J. A. Ortega-Yag\"ues, Concircular helices and concircular surfaces in Euclidean 3-space $\r^3$, Hacet. J. Math. Stat. 52 (2023), 995-1005.

\bibitem{loy5} P. Lucas, J. A. Ortega-Yag\"ues, Concircular hypersurfaces and concircular helices in space forms, Mediterr. J. Math. 20 (2023), Paper No. 320, 13 pp.

\bibitem{mam2} C. A. Mantica, L. G. Molinari, On the Weyl and Ricci tensors of generalized Robertson-Walker space-times, J. Math. Phys. 57 (2016), 102502, 6 pp.

\bibitem{mam0} C. A. Mantica, L. G. Molinari, Generalized Robertson-Walker space times, a survey, Int. J. Geom. Methods Mod. Phys. 14 (2017), 1730001.

\bibitem{mam1} C. A. Mantica, L. G. Molinari, Twisted Lorentzian manifolds, a characterization with torse-forming time-like unit vectors, Gen. Relativity Gravitation 49 (2017), Paper No. 51, 7 pp.

\bibitem{mam3} C. A. Mantica, L. G. Molinari, A note on concircular structure space-times, Commun. Korean Math. Soc. 34 (2019), 633-635.

\bibitem{mih0} A. Mihai, I. Mihai, Torse forming vector fields and exterior concurrent vector fields on Riemannian manifolds and applications, J. Geom. Phys. 73 (2013), 200-208.

\bibitem{na} D. M. Naik, Ricci solitons on Riemannian manifolds admitting certain vector field, Ricerche di Matematica 73 (2024), 531-546.


\bibitem{ny} K. Nomizu and K. Yano, On circles and spheres in Riemannian geometry, Math. Ann. 210 (1974), 163-170.


\bibitem{ya1} K. Yano, Concircular geometry. I. Concircular transformations, Proc. Imp. Acad. Tokyo 16 (1940), 195-200.

\bibitem{ya0} K. Yano, On torse forming direction in a Riemannian space, Proc. Imp. Acad. Tokyo 20 (1944), 340-345.

\bibitem{ymy} H. İ. Yoldaş, Ş. E. Meriş, E. Yaşar, On submanifolds of Kenmotsu manifold with torqued
vector field, Hacet. J. Math. Stat. 49(2) (2020), 843-853.

\end{thebibliography}
 \end{document}